\documentclass[11pt]{article}
\usepackage[a4paper]{geometry}
\usepackage{amsthm}
\usepackage{amsmath}
\usepackage{amssymb}
\usepackage{amsfonts}
\usepackage{epsfig,graphicx}
\usepackage{subfigure}
\usepackage{url}

\theoremstyle{definition}
\newtheorem{Corollary}{Corollary}
\newtheorem{Theorem}{Theorem}

\begin{document}

\title{Explicit high-order symplectic integrators for charged particles in general electromagnetic fields}

\author{Molei Tao (\url{mtao@gatech.edu})}

\date{May 13, 2015}

\maketitle

\begin{abstract}
\noindent
This article considers non-relativistic charged particle dynamics in both static and non-static electromagnetic fields, which are governed by nonseparable, possibly time-dependent Hamiltonians. For the first time, explicit symplectic integrators of arbitrary high-orders are constructed for accurate and efficient simulations of such mechanical systems. Performances superior to the standard non-symplectic method of Runge-Kutta are demonstrated on two examples: the first is on the confined motion of a particle in a static toroidal magnetic field used in tokamak; the second is on how time-periodic perturbations to a magnetic field inject energy into a particle via parametric resonance at a specific frequency.
\end{abstract}

\section{Introduction}
Simulations of the possibly chaotic dynamics of charged particles require high accuracies over long time intervals. The conservation of particle energy, which collision cross-section typically depends on, is also preferred (e.g., \cite{mao2011comparison}). These considerations naturally place symplectic integrators as popular candidates of choice\footnote{Note non-symplectic integrators including Runge-Kutta methods (e.g., \cite{butcher2008numerical,MR1227985}) and Boris' method \cite{boris1970relativistic,birdsall2014plasma} have also been popular.}: based on their preservation of phase-space volume (see e.g., reviews in \cite{sanz1992symplectic,MaWe:01,Hairer06}), symplectic integrators nearly preserve the energy of an autonomous mechanical system (by backward analysis; see e.g., \cite{Hairer06}), correctly account for energy injections and removals in non-conservative mechanical systems (see e.g., \cite{ObTa13}), conserve momentum maps (by a discrete Noether's theorem; see e.g., \cite{MaWe:01}), and demonstrate suitable for long time simulations (e.g., \cite{quispel1998volume,Hairer06}) and chaotic systems (e.g., \cite{mclachlan1992accuracy}). In addition, a need for accuracy in the trajectory calls for high-order symplectic integrators. At the same time, as one often estimates statistics by tracking a large ensemble of particles, computational efficiency is critical, and integrators that use explicitly defined updates are desired.

Although by now it is well known how to construct explicit high-order symplectic integrators for separable Hamiltonians (i.e. $H(q,p)=K(p)+V(q)$; see e.g., \cite{yoshida1990construction, suzuki1990fractal, forest1990fourth, wisdom1991symplectic, sanz1992symplectic, quispel1998volume, MaWe:01, mclachlan2002splitting, leimkuhler2004simulating, Hairer06}), a charged particle is governed by, unfortunately, a nonseparable Hamiltonian. Generic symplectic integrators with implicitly defined updates (e.g., \cite{sanz1988runge}) still apply, but they are computationally much more expensive. Remarkably, an explicit 2nd-order symplectic integrator for static magnetic fields was obtained in \cite{wu2003explicit} based on splitting method. While that pioneering work is generalizable to high-order methods, it doesn't extend to time-dependent fields, and there are certain static fields to which it doesn't apply either (see Comparison in section \ref{sec_Method}). Another explicit symplectic method based on generating function was recently described in \cite{Tao15PRE}; it is, however, only 1st-order.

As difficulties in constructing explicit high-order symplectic integrators for general charged particles are well recognized (see e.g., \cite{he2015volume} for a review, and \cite{webb2014symplectic,Zhang201543} for an instance of an unsuccessful but useful effort), alternative approaches that keep a fraction of the structure-preserving properties of symplectic integrators have been proposed. Particularly significant among them is Boris' method \cite{boris1970relativistic} and the class of volume-preserving integrators \cite{he2015volume}, the latter recently constructed and shown to contain the former.

This article solves the aforementioned difficulties by proposing a family of integrators that are explicit, high-order, and fully symplectic. This family is based on a shadowing theorem and Runge-Kutta methods, and called ESSRK (Explicit Symplectic Shadowed Runge-Kutta; suggested pronunciation: 'e-serk). Any even-order-of-accuracy version of ESSRK can explicitly constructed, and it works for both static and time-dependent electromagnetic fields. The method is described in section \ref{sec_Method}, followed by performance tests on a static example in section \ref{sec_example1} (a charged particle in tokamak) and a time-dependent example in section \ref{sec_example2} (magnetic parametric excitation). In these examples, ESSRK demonstrates superior long-time performances: it makes little amplitude errors, and even its phase errors, which are much more significant, are still smaller than errors of a standard non-symplectic method (Runge-Kutta) with the same order of accuracy.

\section{Method and properties}
\label{sec_Method}
A non-relativistic charged particle in a possibly time-dependent electromagnetic field corresponds to the Hamiltonian
\begin{equation}
	H(\textbf{q},\textbf{p},t)=\frac{1}{2 m} \| \textbf{p}-e \textbf{A}(\textbf{q},t) \|_2^2 + e \phi(\textbf{q},t),
	\label{eq_Hamiltonian}
\end{equation}
where $\textbf{q}=[q_1,q_2,q_3]=[x,y,z]$ is the particle's position and $\textbf{p}=[p_1,p_2,p_3]$ is its momentum, $\textbf{A}$ is the magnetic vector potential ($\textbf{B}=\nabla \times \textbf{A}$) and $\phi$ is the electric scalar potential ($\textbf{E}=\nabla \phi$), and $m$ and $e$ are particle mass and charge. The particle dynamics is governed by Hamilton's equation $\dot{\textbf{q}}=\partial H/\partial \textbf{p}$, $\dot{\textbf{p}}=-\partial H/\partial \textbf{q}$.

\paragraph{The flow maps.} To simulate such a system, I adopt a Hamiltonian splitting approach and write $H=H_1+H_2$, where $H_1=\| \textbf{p} \|_2^2/(2m)$ and $H_2=-e\langle \textbf{p},\textbf{A}(\textbf{q},t) \rangle/m + e^2\|\textbf{A}(\textbf{q},t) \|_2^2/(2m)+e\phi(\textbf{q},t)$. Let $h$ be the integration timestep, and $\psi_1(h)$ be the $h$-time flow map of $H_1$, which is given by
\[
	\psi_1(h):[\textbf{q},\textbf{p}] \mapsto [\textbf{Q},\textbf{P}]:= [\textbf{q}+h \textbf{p}/m, \textbf{p}].
\]
The time $t$ to $t+h$ flow map of $H_2$, indicated by $\bar{\psi}_2(t+h,t)$, is not analytically available, but a good approximation $\psi_2(t+h,t)$ can be obtained by using Runge-Kutta approximation for position and then a symplectically shadowed momentum update. More precisely, let $f(\textbf{q},t)=e^2\|\textbf{A}(\textbf{q},t) \|_2^2/(2m)+e\phi(\textbf{q},t)$, and then $H_2$ generates the dynamics
\begin{equation} \begin{cases}
	\dot{\textbf{q}}=-(e/m) \textbf{A}(\textbf{q},t) \\
	\dot{\textbf{p}}=(e/m) \textbf{p} \textbf{A}'(\textbf{q},t)-\nabla f(\textbf{q},t)
\end{cases}, 
\label{eq_H2dynamics}
\end{equation}
where $\textbf{q}$, $\textbf{p}$ and $\textbf{A}$ are viewed as row vectors, $\nabla f$ is a row vector containing three spatial derivatives, and $[\textbf{A}']_{ij}:=\partial A_i / \partial q_j$ (the same conventions will be used throughout this paper for such differential operators). The fact that $\textbf{q}$ dynamics shadow $\textbf{p}$'s is utilized by letting
\begin{align*}
	\textbf{k}_i (\textbf{q}) &= -(e/m) \textbf{A}\left(\textbf{q}+h\sum_{j=1}^{i-1} a_{ij} \textbf{k}_j(\textbf{q}), t+h\sum_{j=1}^{i-1} a_{ij} \right) \\
	\textbf{k}'_i (\textbf{q}) &= -(e/m) \textbf{A}'\left(\textbf{q}+h\sum_{j=1}^{i-1} a_{ij} \textbf{k}_j(\textbf{q}), t+h\sum_{j=1}^{i-1} a_{ij} \right) \left(\textbf{I}+h\sum_{j=1}^{i-1} a_{ij} \textbf{k}'_j(\textbf{q})\right) \\
	\nabla l_i (\textbf{q}) &= \nabla f\left(\textbf{q}+h\sum_{j=1}^{i-1} a_{ij} \textbf{k}_j(\textbf{q}), t+h\sum_{j=1}^{i-1} a_{ij} \right) \left(\textbf{I}+h\sum_{j=1}^{i-1} a_{ij} \textbf{k}'_j(\textbf{q})\right)
\end{align*}
be explicitly computed for $i=1,\cdots, s$, where dependence on $t$ and $h$ are implicitly assumed for notational brevity. Then $\psi_2(t+h,t):[\textbf{q},\textbf{p}] \mapsto [\textbf{Q},\textbf{P}]$ defined by
\begin{align}
	\textbf{Q} &= \textbf{q}+h \sum_{i=1}^s b_i \textbf{k}_i (\textbf{q}) \nonumber\\
	\textbf{P} &= \left( \textbf{p}- h \sum_{i=1}^s b_i \nabla l_i(\textbf{q}) \right) \left(\textbf{I}+h\sum_{i=1}^{s} b_i \textbf{k}'_i(\textbf{q})\right)^{-1}
\label{eq_psi2}
\end{align}
is (i) a symplectic map for any $t$ and $h$, and (ii) an $\mathcal{O}(h^{p+1})$ approximation of $\bar{\psi}_2(t+h,t)$ as long as $s$, $a_{ij}$ and $b_i$ are parameters of a generic $p$-th order Runge-Kutta method (such parameters values exist for arbitrary positive integer $p$; e.g., \cite{butcher2008numerical,MR1227985}).

To prove these properties, let 
$l_i(\textbf{q})=f\left(\textbf{q}+h\sum_{j=1}^{i-1} a_{ij} \textbf{k}_j(\textbf{q}), t+h\sum_{j=1}^{i-1} a_{ij} \right)$,
$\textbf{g}(\textbf{q})=\textbf{q}+h \sum_{i=1}^s b_i \textbf{k}_i (\textbf{q})$, and $c(\textbf{q})=h\sum_{i=1}^{s} b_i l_i(\textbf{q})$, and then \eqref{eq_psi2} can be verified to be $\textbf{Q}=\textbf{g}(\textbf{q}), \textbf{P}=(\textbf{p}-\nabla c(\textbf{q}))\left(\textbf{g}'(\textbf{q})\right)^{-1}$.

(i) For any fixed $t$ and $h$, consider a generating function of 2nd kind $S(\textbf{P},\textbf{q})=\langle \textbf{P},\textbf{g}(\textbf{q}) \rangle + c(\textbf{q})$. $\psi_2$ is symplectic because it corresponds to a canonical transformation $\textbf{Q}=\partial S/\partial \textbf{P}$, $\textbf{p}=\partial S/\partial \textbf{q}$.

(ii) Two observations help quantify the accuracy of $\psi_2$. First, $\textbf{p}$ is shadowed by $\textbf{q}$ dynamics in \eqref{eq_H2dynamics}, in the sense that if $\textbf{q}(t)$ is exactly available then $\textbf{p}(t)$ can be explicitly obtained (note this is nontrivial: even though $\textbf{p}$ satisfies a linear equation given $\textbf{q}$, the linear coefficient $-(e/m)\textbf{A}'(\textbf{q},t)$ is a time-dependent matrix, which makes a closed-form solution not obvious). Second, note $\textbf{q}$ is approximated by a standard Runge-Kutta method, which introduces an $\mathcal{O}(h^{p+1})$ error, and $\textbf{p}$, obtained via shadowing, will have an error at the same order. These observations are made precise in the appendix.

\paragraph{The integrator based on flow composition.} It is known that the flow of a time independent Hamiltonian $H=H_1+H_2$ can be approximated to arbitrary high order by a careful alternating composition of flows of $H_1$ and $H_2$ (e.g., \cite{suzuki1990fractal, yoshida1990construction, mclachlan2002splitting, Hairer06}). This powerful tool extends to the time-dependent system \eqref{eq_Hamiltonian} (provable upon the introduction of a dummy time variable). Specifically, let $\Psi(t+h,t)$ be the flow map of $H$, and then it has a 2nd-order approximation given by
$\Psi(t+h,t)=\Theta_2(t+h,t)+\mathcal{O}(h^3)$, where
\begin{equation}
	\Theta_2(t+h,t):= \psi_1\left(t+h,t+h/2\right) \circ \bar{\psi}_2 (t+h,t) \circ \psi_1\left(t+h/2,t \right).
	\label{eq_2ndOrderComposition}
\end{equation}
Furthermore, a $(p+2)^{th}$-order approximation can be constructed from $p^{th}$-order via
\begin{equation}
	\Theta_{p+2}(t+h,t) = \Theta_p\left(t+h,t+(1-\gamma_p)h\right) \circ \Theta_p\left(t+(1-\gamma_p) h,t+\gamma_p h\right) \circ \Theta_p\left(t+\gamma_p h,t\right),
	\label{eq_highOrderComposition}
\end{equation}
where $\gamma_p=1/(2-2^{1/(p+1)})$. Hence, given an arbitrary even $p$, $\Theta_{p}(t+h,t)$ that satisfies $\Psi(t+h,t)=\Theta_{p}(t+h,t)+\mathcal{O}(h^{p+1})$ can be iteratively constructed.

The problem is $\bar{\psi}_2$ is unavailable. However, $\psi_2$, obtained by shadowed Runge-Kutta \eqref{eq_psi2}, is an $p$-th order approximation $\psi_2$. Replacing $\bar{\psi}_2$ by $\psi_2$ in $\Theta_p$ leads to an additional error, but the total error remains $\mathcal{O}(h^{p+1})$. This way, high-order integrators for \eqref{eq_Hamiltonian} are constructed, and they are symplectic because both $\psi_1$ and $\psi_2$ are symplectic and hence so are their compositions.

The resulting integrators are called ESSRK. A recommended 4th-order ESSRK updates from $[\textbf{q}_n,\textbf{p}_n]$ at time $t_n$ to $[\textbf{q}_{n+1},\textbf{p}_{n+1}]$ at time $t_{n+1}=t_n+h$, based on
\begin{equation} \begin{cases}
	\textbf{q}_{n,1} &= \textbf{q}_n+(\gamma h /2) \textbf{p}_n/m \\
	[\textbf{q}_{n,2},\textbf{p}_{n,2}] &= \psi_2(t_n+\gamma h, t_n)[\textbf{q}_{n,1},\textbf{p}_n] \\
	\textbf{q}_{n,3} &= \textbf{q}_{n,2}+(h/2-\gamma h/2) \textbf{p}_{n,2}/m \\
	[\textbf{q}_{n,4},\textbf{p}_{n,4}] &= \psi_2(t_n+(1-\gamma)h, t_n+\gamma h)[\textbf{q}_{n,3},\textbf{p}_{n,2}] \\
	\textbf{q}_{n,5} &= \textbf{q}_{n,4}+(h/2-\gamma h/2) \textbf{p}_{n,4}/m \\
	[\textbf{q}_{n,6},\textbf{p}_{n+1}] &= \psi_2(t_n+h, t_n+(1-\gamma) h)[\textbf{q}_{n,5},\textbf{p}_{n,4}] \\
	\textbf{q}_{n+1} &= \textbf{q}_{n,6}+(\gamma h /2) \textbf{p}_{n+1}/m
\end{cases},
\label{eq_ESSRK4th}
\end{equation}
where $\gamma=1/(2-2^{1/3})$, and $\psi_2$ uses $s=4$; $b_1=\frac{1}{6}, b_2=\frac{2}{6}, b_3=\frac{2}{6}, b_4=\frac{1}{6}$; $a_{21}=\frac{1}{2}, a_{32}=\frac{1}{2}, a_{43}=1$ and other $a_{ij}$'s are 0. Note adjacent $\psi_1$'s in the composition \eqref{eq_highOrderComposition} have been absorbed into a single substep since $\psi_1$ forms a semigroup, and therefore $\Theta_4$ consists of four $\psi_1$ substeps and three $\psi_2$ updates, each of which involves four stages.

In general, ESSRK of arbitrary order $p$ can be viewed as an alternating composition
\[
	\psi_1(s_{l_p+1},s_{l_p}) \circ \psi_2(\tau_{l_p},\tau_{l_p-1}) \circ \psi_1(s_{l_p},s_{l_p-1}) \circ \cdots \circ \psi_1(s_2,s_1) \circ \psi_2(\tau_1,\tau_0) \circ \psi_1(s_1,s_0),
\]
where $s_0=\tau_0=t_n$, $s_{l_p+1}=\tau_{l_p}=t_{n+1}$, and other nodes can be computed using \eqref{eq_highOrderComposition} for arbitrary $p$ (with $l_{p+2}=3l_p-2$). Note when $p$ is large it is possible to obtain alternative values of $s$ and $\tau$'s with a smaller $l_p$, and hence a reduced number of substeps; this is based on order conditions obtained from free Lie algebra theory (see e.g., \cite{mclachlan1995numerical, murua1999order, mclachlan2002splitting}).

\paragraph{Comparison with an existing method.} There is one existing approach for explicit high-order symplectic integrations of charged particles dynamics (see \cite{wu2003explicit,forest2006geometric}). Although remarkable (\cite{Dragt2015Lie} Chapter 12.9), that approach only works for static electromagnetic fields, while ESSRK works for time-dependent fields as well.

In addition, there are static fields to which the existing method doesn't apply. More specifically, the existing method is based on splitting and canonical transformation --- the Hamiltonian $H=\|\textbf{p}-e \textbf{A}(\textbf{q})\|^2/(2m)+e \phi(\textbf{q})$ is decomposed as a sum of $H_i=(p_i-e A_i(\textbf{q}))^2/(2m)$, $i=1,2,3$ and $H_4=e \phi(\textbf{q})$, and the flow of $H$ is approximated by composing the flows $\varphi_i(h)$ of $H_i$ ($i=1,\cdots,4$). $\varphi_1$ can be obtained by introducing a canonical transformation $\textbf{Q}=\textbf{q}$, $\textbf{P}=\textbf{p}-\frac{\partial S_1}{\partial \textbf{q}}$, where
\begin{equation}
	S_1(\textbf{q})=e \int^{q_1} A_1(\hat{q}_1,q_2,q_3) d \hat{q}_1. \label{eq_existingMethodGeneratingFct}
\end{equation}
Under this transformation, $H_1$ becomes $H_1(\textbf{Q},\textbf{P})=P_1^2/2m$, and its flow corresponds to a simple shift in $Q_1$. $\varphi_2$ and $\varphi_3$ can be analogously obtained, and $\varphi_4$ is simply a momentum shift. However, the integral in \eqref{eq_existingMethodGeneratingFct} is not always obtainable in closed-form --- a simple counterexample is $A_1(q_1,q_2,q_3)=\exp(-q_1^2)q_2$, in which case $\partial_{q_2} S_1$ cannot be computed in closed-form and hence the existing method doesn't apply. ESSRK doesn't have this problem because it requires only the governing equation ($\nabla\phi$, $\textbf{A}'$ and $\textbf{A}$).

A third difference is computational, in that medium-order ESSRK integrators involve less substeps. For instance, the existing method stated in \cite{wu2003explicit} is 2nd-order, obtained via the composition $\varphi_1(h/2) \circ \varphi_2 (h/2) \circ \varphi_3 (h/2) \circ \varphi_4 (h) \circ \varphi_3 (h/2) \circ \varphi_2 (h/2) \circ \varphi_1 (h/2)$. It involves 7 substeps. A 4th-order generalization can also be obtained using \eqref{eq_highOrderComposition}, and it will involve 19 substeps. In comparison, a 2nd-order ESSRK involves 4 substeps (1+2+1, the middle 2 corresponds to the two stages in a RK2 update), and the 4th-order ESSRK \eqref{eq_ESSRK4th} involves 1+4+1+4+1+4+1=16 substeps.

\paragraph{Generalization.} ESSRK straightforwardly extends to $N$ charged particles with an interacting potential that only depends on particle positions, i.e.
\[
	H(\textbf{q}_1,\textbf{p}_1,\cdots,\textbf{q}_N,\textbf{p}_N,t)=\sum_{j=1}^N \left( \frac{1}{2 m_j}\|\textbf{p}_j-e_j \textbf{A}_j(\textbf{q}_j,t)\|^2+e_j \phi_j(\textbf{q}_j,t) \right) + V(\textbf{q}_1,\cdots,\textbf{q}_N,t).
\]
The computational cost scales with $N$ without overhead (except for the unavoidable evaluation of $\nabla V$). This is because $\phi_j(\textbf{q}_j,t)$'s and $V(\textbf{q}_1,\cdots,\textbf{q}_N,t)$ can be absorbed into a single function $\phi(\textbf{q})$, whose contribution is accounted for by a single $\psi_2$ substep, and the matrix to be inverted in \eqref{eq_psi2} is block-diagonal.

\section{Example 1: a particle confined by a toroidal field}
\label{sec_example1}

Consider a charged particle in a toroidal magnetic field used in tokamak. I follow the model in \cite{cambon2014chaotic}, which uses the static magnetic field
\[
	\textbf{B}(r,\theta,\phi)=\frac{B_0 R}{R+r \cos(\theta)} \left( \hat{\textbf{e}}_\phi + \frac{r}{Q R} \hat{\textbf{e}}_\theta \right),
\]
where $r,\theta,\phi$ are toroidal coordinates, $R$, $B_0$, and $Q$ are constants (note the safety factor is denoted by $Q$ instead of $q$ used in \cite{cambon2014chaotic} to avoid confusion with the position variable). The corresponding vector potential in Cartesian coordinates under Coulomb gauge can be computed as
\[
	\textbf{A}(x,y,z)=B_0 \begin{bmatrix} -\frac{(\sqrt{x^2+y^2}-R)^2+z^2}{2 Q(x^2+y^2)} y, &
										   \frac{(\sqrt{x^2+y^2}-R)^2+z^2}{2 Q(x^2+y^2)} x, &
										   -R \log \left( \frac{\sqrt{x^2+y^2}}{R} \right) \end{bmatrix}.
\]
To demonstrate the applicability of ESSRK, I also add an electric field with scalar potential $\phi(x,y,z)=-E_0\cos(z)$.

\begin{figure}[htb]
\centering
\footnotesize
\subfigure[Runge-Kutta with $h=0.5$]{
\includegraphics[width=0.31\textwidth]{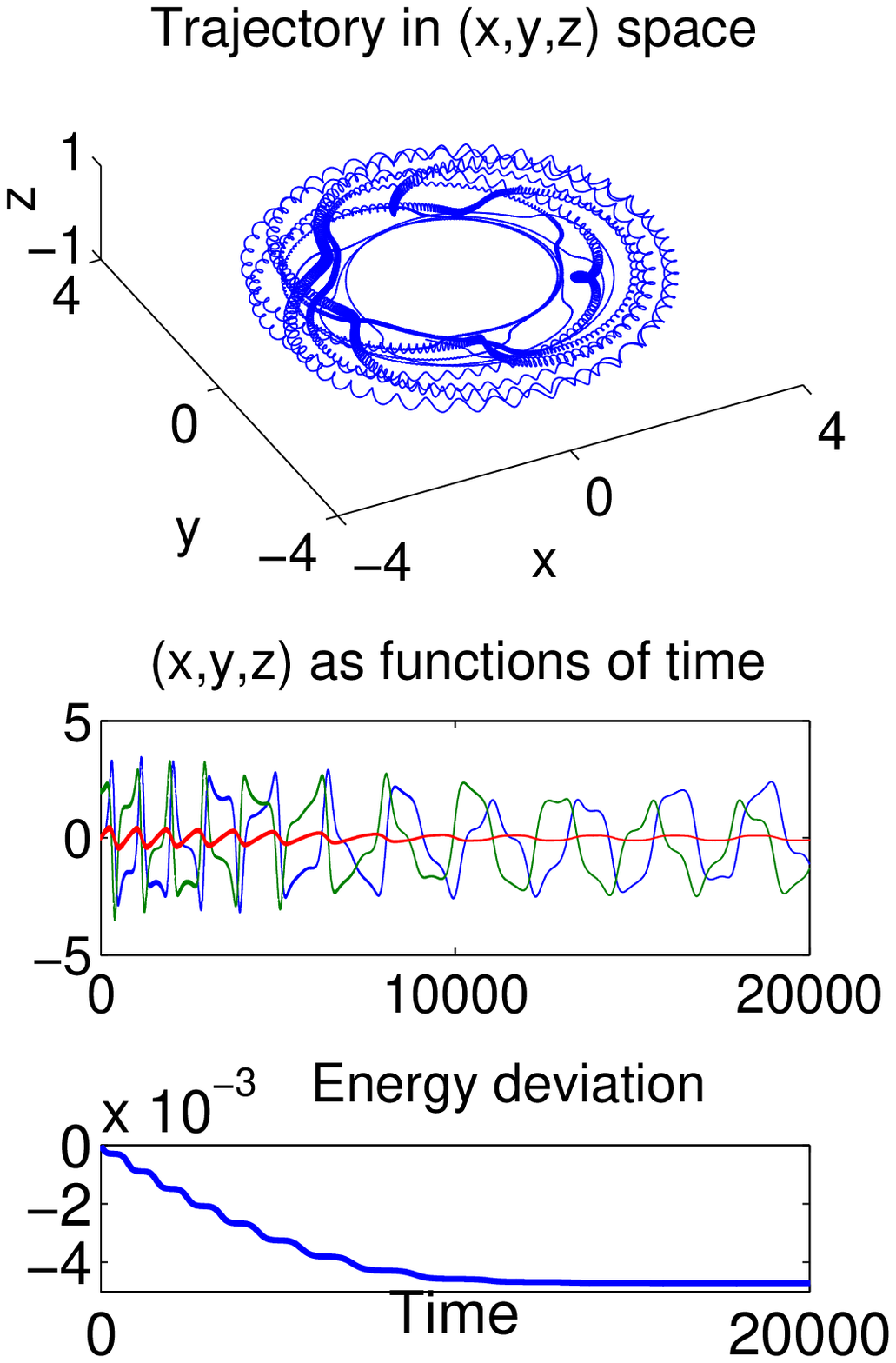}
\label{fig_tokamark_RKcoarse}
}
\subfigure[ESSRK with $h=0.5$]{
\includegraphics[width=0.31\textwidth]{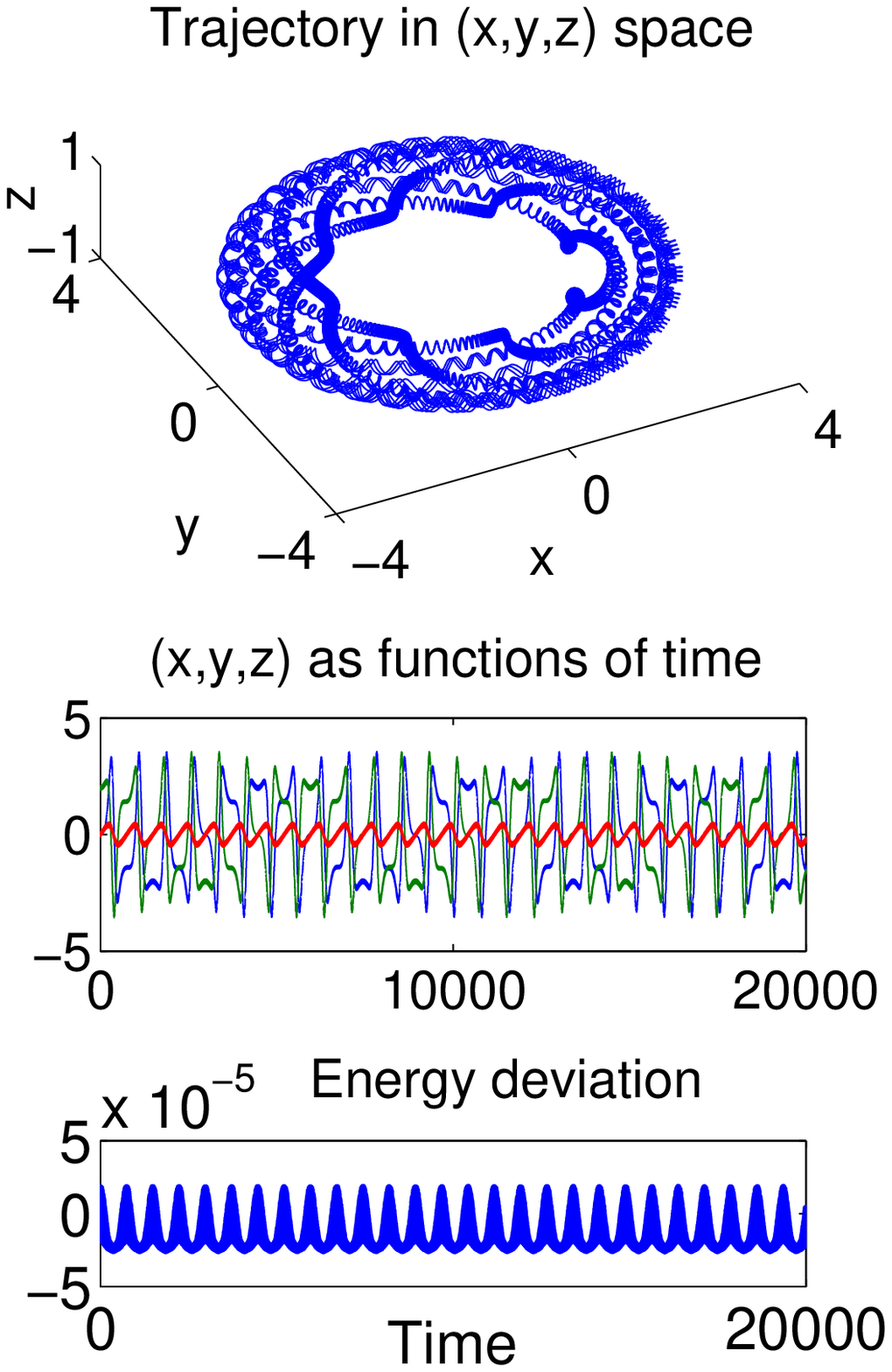}
\label{fig_tokamark_ESSRKcoarse}
}
\subfigure[Runge-Kutta with $h=0.01$ (benchmark)]{
\includegraphics[width=0.31\textwidth]{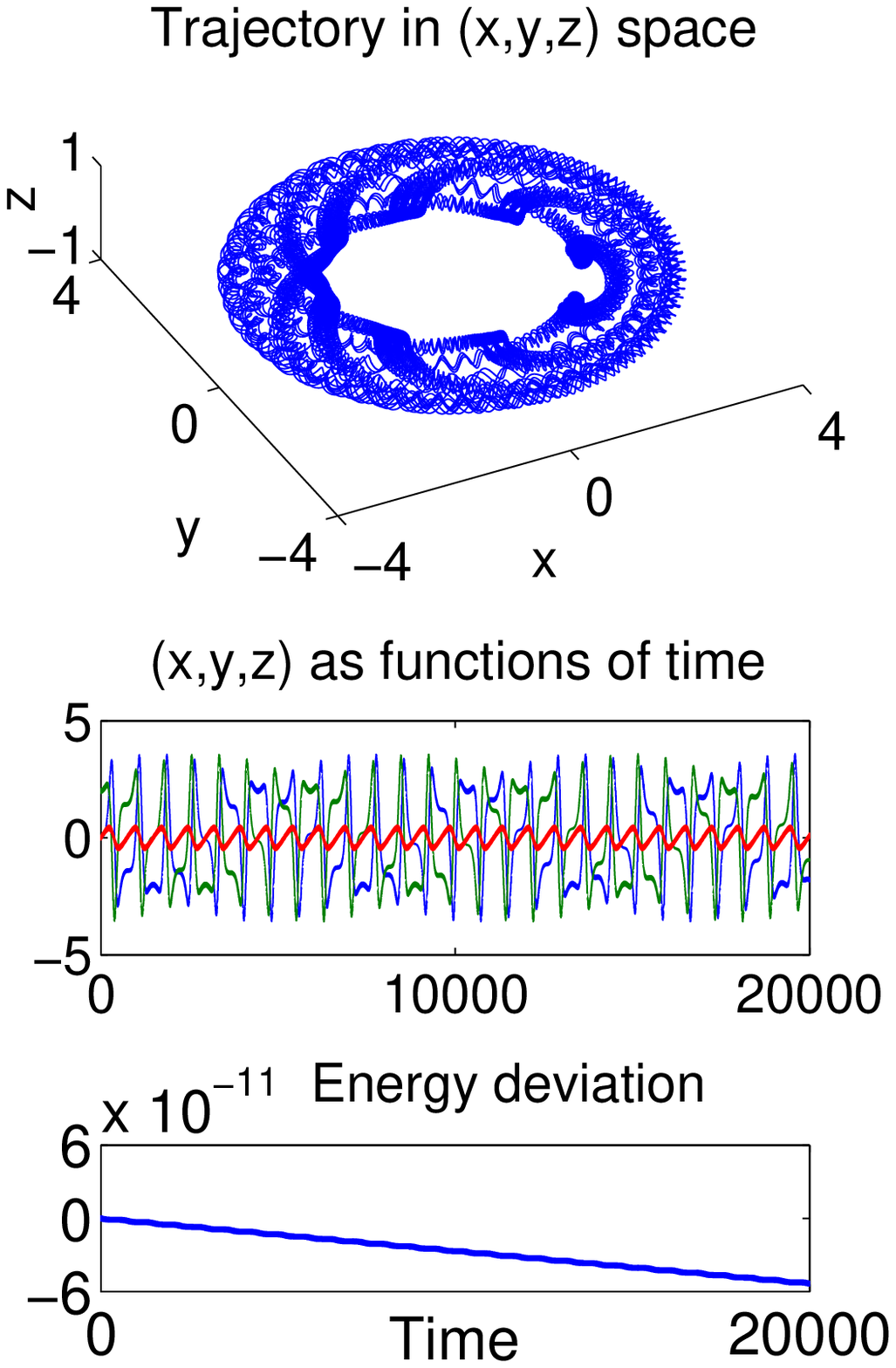}
\label{fig_tokamark_RKfine}
}
\caption{\footnotesize Simulations of a charged particle in tokamak. For simplicity in demonstration, adopt unitless convention and let charge $e=1$ and mass $m=1$. Use parameters $B_0=1$, $E_0=10^{-2}$, $R=2$, $Q=5$ and initial condition $q(0)=[0, 2.1, 0], p(0)=[0, 0, 0]$.}
\label{fig_tokamak}
\end{figure}

Figure \ref{fig_tokamak} compares the simulation by the 4th-order ESSRK \eqref{eq_ESSRK4th} with the standard 4th-order Runge-Kutta. Although both methods are 4th-order, standard RK loses accuracy in a long time simulation, and its lack of symplecticity results in numerical dissipation. ESSRK as a symplectic method has much better long time performance, and this is observed even when the timestep is large. Note symplecticity doesn't mean the elimination of numerical errors, and a scrutinized comparison between row 2 columns 2 and 3 shows that a large timestep still leads to phase errors; such errors are, of course, suppressed when ESSRK employs a small timestep (results not shown).

\section{Example 2: a charged particle in parametric resonance} 
\label{sec_example2}

Consider a spatial-homogeneous magnetic field with a fixed direction and periodically perturbed amplitude, which is assumed without loss of generality to be $B(t)=1+\epsilon \sin(\omega t)$ (along with $e=1,m=1$). Choose $x$-$y$ plane perpendicular to the magnetic field and Coulomb gauge so that $\textbf{A}(\textbf{q},t)=B(t)[q_2,-q_1,0]/2$. Consider $\textbf{q}(0)=[0, 2.1, 0]$ and $\textbf{p}(0)=[0, 0, 0]$.

In this simple example, amplitude and phase errors of a numerical simulation can be identified: analogous to guiding-center reduction (e.g., \cite{littlejohn1979guiding}), the velocity $\textbf{v}=\textbf{p}-\textbf{A}(\textbf{q},t)$ is represented in polar coordinates $v_1=v\sin\theta$, $v_2=v\cos\theta$, and then $v$ and $\theta$ respectively correspond to the slowly-varying amplitude and the fast-varying phase of the velocity oscillations.

In addition, temporal homogenization theory \cite{tao14temporal} helps show that when $\omega=1$ the particle experiences parametric resonance: let $E(t)=H(\textbf{q}(t),\textbf{p}(t))$, then $E(t)\approx e^{\epsilon t/2}(E(0)+\mathcal{O}(\epsilon))$ for the above initial condition. Note $t$ has to be large for the time-dependent magnetic field to pump into the particle an observable amount of energy.

\begin{figure}[htb]
\centering
\includegraphics[width=\textwidth]{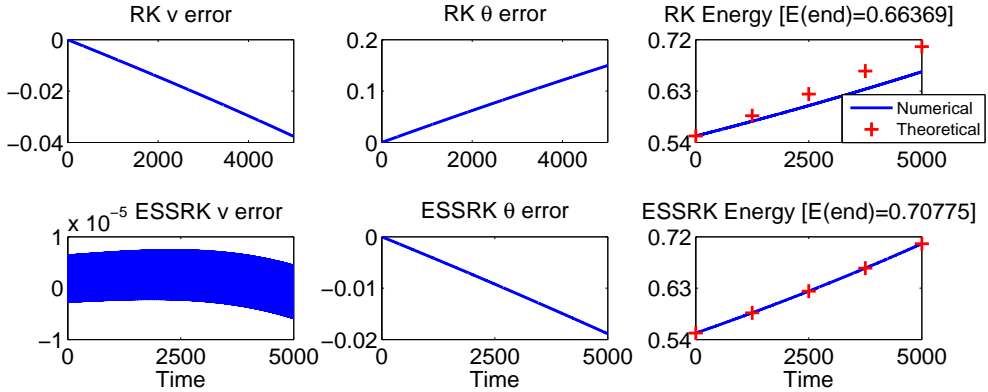}
\caption{\footnotesize Simulations of a charged particle in parametric resonance. $h=0.25$, $\epsilon=10^{-4}$, $T=\epsilon^{-1}/2$. Errors were obtained by comparing to a benchmark simulation by RK4 with $h=0.001$.}
\label{fig_PR}
\end{figure}

The long-time performances of the 4th-order ESSRK \eqref{eq_ESSRK4th} and the standard 4th-order Runge-Kutta are compared in terms of amplitude and phase errors in figure \ref{fig_PR} columns 1 and 2. ESSRK produces much smaller amplitude error due to its structure preservation property. At the same time, phase error of ESSRK is more significant, but it is nevertheless still one order of magnitude smaller than that of RK. In addition, numerically obtained particle energies are compared in figure \ref{fig_PR} column 3. Temporal homogenization shows that $E(\text{end})\approx 0.7078$, which agrees well with the ESSRK result. On the contrary, the non-symplectic simulation by RK is inaccurate, because while the particle gains energy from the perturbation, a large proportion of the gain is drained as a dissipative numerical artifact.

\section{Acknowledgment}
The author thanks Jonathan Goodman for introducing him to this interesting problem.

\section{Appendix: proof of the arbitrary order of accuracy}
\begin{Theorem}[Shadowing]
	Consider $\begin{cases} \dot{\textbf{q}}=-(e/m) \textbf{A}(\textbf{q},t) \\ \dot{\textbf{p}}=(e/m) \textbf{p} \textbf{A}'(\textbf{q},t)-\nabla f(\textbf{q},t) \end{cases}$. Define the flow map $\bar{\textbf{g}}(t,t_0,\textbf{x}):=\textbf{q}(t)$ given $\textbf{q}(t_0)=\textbf{x}$, and let $\bar{c}(t,t_0,\textbf{x})=\int_{t_0}^t f(\textbf{q}(\tau),\tau)d\tau$. Then
	\begin{equation}
		\textbf{p}(t)=\left( \textbf{p}(t_0)-\frac{\partial \bar{c}}{\partial \textbf{x}}(t,t_0,\textbf{x}) \right) \left[ \frac{\partial \bar{\textbf{g}}}{\partial \textbf{x}}(t,t_0,\textbf{x}) \right]^{-1}.
	\label{eq_shadowEqSln}
	\end{equation}
\end{Theorem}
\begin{proof}
	By the definition of $\bar{\textbf{g}}$, $\partial \bar{\textbf{g}} / \partial t = -(e/m) \textbf{A}(\bar{\textbf{g}},t)$, and thus
	\[
		\frac{\partial}{\partial t} \left( \frac{\partial \bar{\textbf{g}}}{\partial \textbf{x}} \right) =
		\frac{\partial}{\partial \textbf{x}} \left( \frac{\partial \bar{\textbf{g}}}{\partial t} \right) =
		-\frac{e}{m} \textbf{A}' \frac{\partial \bar{\textbf{g}}}{\partial \textbf{x}}.
	\]
	Then,
	\begin{align*}
		&\quad \frac{\partial}{ \partial t}\left(\textbf{p}(t) \frac{\partial \bar{\textbf{g}}}{\partial \textbf{x}}(t,t_0,\textbf{x})+ \frac{\partial \bar{c}}{\partial \textbf{x}}(t,t_0,\textbf{x}) \right) \\
		&= \frac{e}{m} \textbf{p} \textbf{A}' \frac{\partial \bar{\textbf{g}}}{\partial \textbf{x}}- \nabla f \frac{\partial \bar{\textbf{g}}}{\partial \textbf{x}} - \frac{e}{m} \textbf{p} \textbf{A}'\frac{\partial \bar{\textbf{g}}}{\partial \textbf{x}}+ \frac{\partial}{\partial \textbf{x}} f(\bar{\textbf{g}}(t,t_0,\textbf{x}),t) \\
		&= \frac{e}{m} \textbf{p} \textbf{A}' \frac{\partial \bar{\textbf{g}}}{\partial \textbf{x}}- \nabla f \frac{\partial \bar{\textbf{g}}}{\partial \textbf{x}} - \frac{e}{m} \textbf{p} \textbf{A}'\frac{\partial \bar{\textbf{g}}}{\partial \textbf{x}}+ \nabla f \frac{\partial \bar{\textbf{g}}}{\partial \textbf{x}} = 0.
	\end{align*}
	Therefore,
	\[
		\textbf{p}(t) \frac{\partial \bar{\textbf{g}}}{\partial \textbf{x}}(t,t_0,\textbf{x})+ \frac{\partial \bar{c}}{\partial \textbf{x}}(t,t_0,\textbf{x})=\textbf{p}(t_0) \frac{\partial \bar{\textbf{g}}}{\partial \textbf{x}}(t_0,t_0,\textbf{x})+ \frac{\partial \bar{c}}{\partial \textbf{x}}(t_0,t_0,\textbf{x}).
	\]
	Since $\bar{\textbf{g}}(t_0,t_0,\textbf{x})=\textbf{x}$ and $\bar{c}(t_0,t_0,\textbf{x})=0$, this equality produces \eqref{eq_shadowEqSln} after rearranging terms.
\end{proof}

\begin{Corollary}[$\psi_2$ order of accuracy]
	Assume $s$ is the number of stages and $a_{ij}$, $b_i$ are coefficients of a $p$-th order Runge-Kutta method, then the update \eqref{eq_psi2} has $\mathcal{O}(h^{p+1})$ local truncation error.
\end{Corollary}
\begin{proof}
	By the definition of a $p$-th order Runge-Kutta method, $\textbf{g}(\textbf{q}(t))=\textbf{q}(t+h)+\mathcal{O}(h^{p+1})$, and thus the position update has $(p+1)$-order truncation error.
	
	It can also be shown that $c(\textbf{q}(t)):=h\sum_{i=1}^s b_i l_(\textbf{q})=\int_t^{t+h} f(\textbf{q}(\tau)) d\tau + \mathcal{O}(h^{p+1})$ by considering an auxiliary system $\begin{cases} \dot{\textbf{q}}=-(e/m)\textbf{A}(\textbf{q},t) \\ \dot{z}=f(\textbf{q},t) \end{cases}$; one-step update of the same Runge-Kutta method applied to this augmented system leads to $\textbf{g}(\textbf{q}(t)) =\bar{\textbf{g}}(t+h,t, \textbf{q}(t))+\mathcal{O}(h^{p+1})$ and
	\[
		h\sum_{i=1}^s b_i l_i(\textbf{q}(t))=z(t+h)-z(t)+\mathcal{O}(h^{p+1}) = \int_t^{t+h} f(\textbf{q}(\tau),\tau) d\tau +\mathcal{O}(h^{p+1}),
	\]
	which is the same as $c(\textbf{q}(t))=\bar{c}(t+h,t, \textbf{q}(t)) + \mathcal{O}(h^{p+1})$. Consequently, the momentum update \eqref{eq_psi2}, in which $[\textbf{q}, \textbf{p}]=[\textbf{q}(t), \textbf{p}(t)]$ and $[\textbf{Q}, \textbf{P}]=[\textbf{q}(t+h), \textbf{p}(t+h)]$, satisfies
	\[\
		\textbf{P}=(\textbf{p}-\nabla c(\textbf{q})) \left(\textbf{g}'(\textbf{q})\right)^{-1} \\
		= \left( \textbf{p}-\frac{\partial \bar{c}}{\partial \textbf{q}}(t+h,t,\textbf{q}) \right) \left[ \frac{\partial \bar{\textbf{g}}}{\partial \textbf{q}}(t+h,t,\textbf{q}) \right]^{-1} + \mathcal{O}(h^{p+1}),
	\]
	where the last estimate uses regularities of $\bar{c}$ and $\bar{\textbf{g}}$ in $\textbf{q}$, which are ensured by Gronwall's lemma. This shows that the momentum truncation error is also of order $p+1$.
\end{proof}

\bibliographystyle{unsrt}
\small
\bibliography{molei23}

\end{document}